\documentclass[12pt]{article}
\usepackage{amssymb}
\usepackage{epsfig}
\usepackage{color}

\normalbaselines

\newtheorem{Theorem}{Theorem}[section]

\newtheorem{Proposition}[Theorem]{Proposition}

\newtheorem{Lemma}[Theorem]{Lemma}

\newtheorem{Corollary}[Theorem]{Corollary}

\newtheorem{Remark}[Theorem]{Remark}

\newcommand{\RR}{{{\rm I} \kern -.15em {\rm R} }}

\newcommand{\C}{{{\rm l} \kern -.42em {\rm C} }}

\newcommand{\nat}{{{\rm I} \kern -.15em {\rm N} }}
\def\be{\begin{equation}}
\def\ee{\end{equation}}
\def\beq{\begin{eqnarray}}
\def\eeq{\end{eqnarray}}
\def\beqs{\begin{eqnarray*}}
\def\eeqs{\end{eqnarray*}}

\newcommand{\caH}{{\mathcal H}}
\newcommand{\caB}{{\mathcal B}}
\newcommand{\caA}{{\mathcal A}}

\newcommand{\caV}{{\mathcal V}}

\newcommand{\bt}{\begin{Theorem}}
\newcommand{\et}{\end{Theorem}}
\newcommand{\br}{\begin{Remark}}
\newcommand{\er}{\end{Remark}}
\newcommand{\bc}{\begin{Corollary}}
\newcommand{\ec}{\end{Corollary}}
\newcommand{\bl}{\begin{Lemma}}
\newcommand{\el}{\end{Lemma}}
\newcommand{\bd}{\begin{definition}}
\newcommand{\ed}{\end{definition}}
\newenvironment{proof}[1][Proof]{\textbf{#1.} }{\ \rule{0.5em}{0.5em}}

\title{Well-posedness and stability results for nonlinear abstract evolution equations with time delays}
\author{{\sc Serge Nicaise}
\\Universit\'e de Valenciennes et du Hainaut Cambr\'{e}sis\\
ISTV--LAMAV\\
59313 Valenciennes Cedex 9, France
\\\\
{\sc Cristina Pignotti}
\\Dipartimento di Ingegneria e Scienze dell'Informazione e Matematica\\
 Universit\`{a} di L'Aquila\\
Via Vetoio, Loc. Coppito, 67010 L'Aquila Italy}

\date{}

\begin{document}

\textwidth=160 mm

\textheight=225mm

\parindent=8mm

\frenchspacing

\maketitle

\begin{abstract}
We consider abstract evolution equations with a nonlinear term depending on the state and on delayed states.
We show that, if the $C_0$-semigroup describing the linear part of the model is exponentially stable, then the whole system retains this  property
under some Lipschitz continuity assumptions on the nonlinearity. More precisely, we give
a general exponential decay estimate for small time delays if the nonlinear term is globally Lipschitz and an exponential decay estimate for solutions starting from small data when the nonlinearity is only locally Lipschitz
and the linear part is a negative selfadjoint operator. In the latter case we do not need any restriction on the size of the time delays. In  both cases, concrete examples  are presented that illustrate our abstract results.
\end{abstract}

\vspace{5 mm}

\def\qed{\hbox{\hskip 6pt\vrule width6pt
height7pt
depth1pt  \hskip1pt}\bigskip}



\section{Introduction}
\label{intro}\hspace{5mm}

\setcounter{equation}{0}

Let $\caH$ be a fixed Hilbert space with inner product $(\cdot,\cdot)_\caH$ and norm $\|\cdot\|_\caH$, and consider an operator $\caA$ from $\caH$ into itself that generates a $C_0$-semigroup  $(S(t))_{t\geq 0}$ that is exponentially stable, i.e.,
there exist two positive constants $M$ and $\omega$ such that
\begin{equation}\label{assumpexpdecay}
\|S(t)\|_{{\mathcal L}(\caH)}\leq M e^{-\omega t},\ \  \forall\  t\geq 0,
\end{equation}
where, as usual, ${\mathcal L}(\caH)$ denotes the space of  bounded linear
operators from $\caH$ into itself.

We consider the evolution equation

 \begin{eqnarray}\label{abstract}
\left\{
\begin{array}{ll}
& U_{t}(t) = \caA U(t)+\sum_{i=1}^I F_i(U(t), U(t-\tau_i)) \quad \mbox{\rm in }
(0,+\infty),\\
& U(t-\tau)=U_0(t), \quad \forall t\in (0,\tau],
\end{array}
\right.
\end{eqnarray}
where $I$ is a positive natural number and $\tau_i >0, i=1,\cdots, I,$ are   time delays.
Without loss of generality we can suppose that the delays are different from each other and that
\[
\tau_{i}<\tau=\tau_1, \ \forall i=2,\cdots, I.
\]
The nonlinear terms $F_i: \caH\times\caH\rightarrow\caH$  satisfy some Lipschitz conditions, while the initial datum $U_0$ satisfies $U_0\in C([0,\tau]; \caH)$. We denote by $U^0$ the initial datum at $t=0,$ namely
\begin{equation}\label{datoinzero}
U^0=U_0(\tau )\in\caH\,.
\end{equation}

Time delay effects frequently appear in many practical applications and physical models. On the other hand, it is well-known (cfr. \cite{Batkai,Datko,DLP,NPSicon06,XYL})
that they can induce some instability.
Then, we are interested in giving an exponential stability result for such a problem under a suitable smallness condition on the delay $\tau$ or on the initial data.
In \cite{JEE15} we have studied the case of linear dependency on the delayed state $U(t-\tau )$, i.e., we have considered the model
 \begin{eqnarray}\label{abstractJEE}
\left\{
\begin{array}{ll}
& U_{t}(t) = \caA U(t)+G(U(t))+k\caB U(t-\tau) \quad \mbox{\rm in }
(0,+\infty),\\
& U(0)=U^0,\   \caB U(t-\tau)=g(t), \quad \forall t\in (0,\tau),
\end{array}
\right.
\end{eqnarray}
where $\caB$ is a fixed bounded operator from $\caH$ into itself, $G: \caH\rightarrow\caH$  satisfies some Lipschitz conditions, the initial datum $U^0$ belongs to $\caH$ and $g\in C([0,\tau]; \caH)$.

For some concrete examples, mainly for $G\equiv 0,$ it was known
that the above problem, under certain {\sl smallness} conditions on the delay feedback $k\caB,$
is exponentially stable,
the proof being from time to time quite technical
because some observability inequalities or perturbation methods are used.
For instance the case of wave equation with interior feedback delay and dissipative boundary condition was considered in \cite{ANP10} by constructing suitable Lyapunov functionals while the case of a locally damped wave equation with distributed delay has been analyzed in \cite{SCL12} by introducing an auxiliary model easier to deal with
and using a perturbative argument. The last approach has then been extended to  a general
class of second order evolution equations in \cite{MCSS}. The case of a Timoshenko system with delay has been studied in \cite{Said} by using appropriate Lyapunov functionals. Also wave type equations with viscoelastic dissipative damping and time delay feedback have been considered (see e.g. \cite{guesmia, AlNP}). Moreover we quote the book   \cite{Batkai} for several examples, also in the parabolic case.

Therefore, in \cite{JEE15} our main goal was to furnish a general stability result based on a direct and more simple proof obtained by using the so-called
  Duhamel's formula (or variation of parameters formula). Indeed we proved an exponential stability result under a suitable condition between the constant $k$ and the constants $M, \omega, \tau,$ the nonlinear term $G$ and the norm of the bounded delay operator $\caB\,.$

Now, we extend  the results in \cite{JEE15} by analyzing the more  general model (\ref{abstract}). In particular, while in \cite{JEE15} we have considered only linear dependency on the delayed state $U(t-\tau )$ (see (\ref{abstractJEE})), we now consider a nonlinear dependency on the state and also on the delayed states.
We first treat the case when $F_i,$ $i=1,\dots, I,$   are globally Lipschitz continuous and give an exponential stability result under a smallness assumption on the time delays. The proof, when adapted to model (\ref{abstractJEE}), is easier with respect to the one that we gave in \cite{JEE15} for the globally Lipschitz case. However, the result from  \cite{JEE15} allowed a bit larger size on the time delay (see Remark \ref{compareJEE}) in order to have an exponential stability estimate for problem (\ref{abstractJEE}).

Then, we assume that the nonlinearity
is only  locally Lipschitz continuous (with a stronger topology) and we prove global existence and an exponential decay estimates for small initial data, independently of the size of the time delays.
In this last case,
we need to restrict ourselves to the case of a negative self-adjoint operator $\caA$. Nevertheless, this last case
has various applications in population dynamics (see section \ref{Examples}).

The paper is organized as follow. In section \ref{pbform} we deal with the globally Lipschitz case; we give a well--posedness result and an exponential stability estimate for small delays. Some concrete examples are also presented.
In section \ref{secgeneralNL} we treat the case where the reaction term is only locally Lipschitz. We prove the local well--posedness and then, for small initial conditions, a global existence result and an exponential decay estimate. Finally, in section \ref{Examples} we give some illustrative applications of the abstract results of section \ref{secgeneralNL} to models in population dynamics. In particular we consider model for single species and competition models for two species.

\section{The case of small delays with $F_i$ globally Lipschitz}\label{pbform}\hspace{5mm}

\setcounter{equation}{0}

In this section we assume that the nonlinear terms are globally Lipschitz and  study
existence and asymptotic behavior of solutions, generalizing Proposition 2.1 and Theorem 2.2 of \cite{JEE15}.

\subsection{Well-posedness and stability estimate}\label{Globally}

Let the functions $F_i,$ $i=1,\dots,I,$ be globally Lipschitz continuous, namely for every $i=1,\dots, I,$
\begin{equation}\label{lip}
\begin{array}{l}
\displaystyle{
\exists \gamma_i >0 \quad \mbox{\rm such\ that}\quad \Vert F_i(U_1, U_2)-F_i(U_1^*, U_2^*)\Vert_\caH
\le \gamma_i (\Vert U_1-U_1^*\Vert_\caH+\Vert U_2-U_2^*\Vert_\caH ),}\\\medskip
\displaystyle {\forall\ (U_1, U_2), (U_1^*, U_2^*)\in \caH\times\caH\,.}
\end{array}
\end{equation}

The following well--posedness result holds.

\begin{Proposition}\label{WellP}
Assume that the functions  $F_i$ satisfy {\rm (\ref{lip})}, for all $i=1,\dots ,I.$
For any initial datum $U_0\in C ([0,\tau]; \caH)$, there exists a unique (mild) solution
$U\in C([0,+\infty), {\mathcal H})$  of problem $(\ref{abstract}).$
Moreover,
\begin{equation}\label{Du}
U(t)=S(t)U^0+\int_0^t S(t-s)\sum_{i=1}^IF_i(U(s), U(s-\tau_i ))\,ds.
\end{equation}
\end{Proposition}

\noindent\begin{proof}
Like in \cite{JEE15}, we use an iterative argument.
Namely in the interval $(0,\tau_{\rm min})$, where $\tau_{\rm min}=\min_{i=1,\ldots, I} \tau_i,$ problem (\ref{abstract}) can be seen as a standard evolution problem
\begin{eqnarray}\label{abstract2}
\left\{
\begin{array}{ll}
& U_{t}(t) = \caA U(t)+g_1(U(t))\quad \mbox{\rm in }
(0,\tau_{\rm min})\\
& U(0)=U^0,
\end{array}
\right.
\end{eqnarray}
where $g_1(U(t))=\sum_{i=1}^IF_i (U(t), U(t-\tau_i))$. Note that the terms $U(t-\tau_i)$ can be regarded as known data for $t\in [0,\tau_{\rm min})\,.$ This problem has a unique
solution
$U\in  C([0,\tau_{\rm min}], {\mathcal H})$ (see Th. 1.2, Ch. 6 of
\cite{pazy})
satisfying
$$
U(t)=S(t)U^0+\int_0^t S(t-s)g_1(U(s))\,ds.
$$
This yields $U(t),$ for $t\in [0,\tau_{\rm min}].$
Therefore on $(\tau_{\rm min},2\tau_{\rm min})$, problem (\ref{abstract}) can be seen as the evolution problem

\begin{eqnarray}\label{abstract3}
\left\{
\begin{array}{ll}
& U_{t}(t) = \caA U(t)+g_2(U(t))\quad \mbox{\rm in }
(\tau_{\rm min}, 2\tau_{\rm min})\\
& U(\tau_{\rm min})=U({\tau_{\rm min}}-),
\end{array}
\right.
\end{eqnarray}
where $g_2(U(t))=\sum_{i=1}^I F_i(U(t),U(t-\tau_i))\,,$ since the terms $U(t-\tau_i )$ can be regarded as  data being known from the first step. Hence, this problem has a unique solution
$U\in C([\tau_{\rm min},2\tau_{\rm min}], {\mathcal H})$ given by
$$
U(t)=S(t-\tau_{\rm min})U({\tau_{\rm min}}-)+\int_{\tau_{\rm min}}^t S(t-s) g_2(U(s))\,ds, \forall t\in 
[\tau_{\rm min},2\tau_{\rm min}].
$$
By iterating this procedure, we obtain a global solution $U$ satisfying (\ref{Du}).\end{proof}

Now we will prove the following exponential stability result.

\begin{Theorem}\label{stab2}
Assume that the functions  $F_i$ satisfy $(\ref{lip})$ and 
\begin{equation}\label{inzero}
F_i(0,0)=0, 
\end{equation}
for all  $i=1,\dots, I.$
With $M, \omega$  from
 $(\ref{assumpexpdecay}),$  
we assume that (see $(\ref{lip})$)
\begin{equation}\label{suF}
 \gamma =\sum_{i=1}^I\gamma_i <\frac {\omega }{2 M}\,.
\end{equation}
If the time delay $\tau$ satisfies
 the {\em smallness} condition
\begin{equation}\label{sutau}
 \tau <\tau_0:=\frac 1 {\omega }\ln \Big (\frac {\omega } {M\gamma }-1 \Big )\,,
 \end{equation}
 then
 there exists $\omega' >0$ such that
the solution $U\in C([0,+\infty), {\mathcal H})$  of problem $(\ref{abstract}),$
with $U_0\in C([0,\tau ]; \caH),$ satisfies
\begin{equation}\label{exponentiald}
\|U(t)\|_\caH\le M e^{-\omega' t} \Big (\|U_0\|_\caH+\sum_{i=1}^I\gamma_i\int_0^{\tau_i} e^{\omega s}\Vert U(s-\tau_i)\Vert_\caH\, ds \Big ),
\quad \forall t\ge 0\,.
\end{equation}
\end{Theorem}
\noindent\begin{proof}
From (\ref{Du}),
we can estimate
\begin{equation}\label{C1}
\begin{array}{l}
\displaystyle{
\Vert U(t)\Vert_\caH \leq M e^{-\omega t} \Big (\|U^0\|_\caH+\sum_{i=1}^I\gamma_i\int_0^t   e^{\omega s}  (\Vert U(s)\Vert_\caH +\Vert U(s-\tau_i )\Vert_\caH )\,ds\Big )}  \\
\hspace {0.7 cm}\le \displaystyle {
M e^{-\omega t} \Big (\|U^0\|_\caH+\gamma
\int_0^t   e^{\omega s}  \Vert U(s)\Vert_\caH\,ds +\sum_{i=1}^I\gamma_i\int_0^t   e^{\omega s}  \Vert U(s-\tau_i )\Vert_\caH \,ds\Big ),\ \  \forall\ t>0\,.}
\end{array}
\end{equation}

Then, for $t\ge \tau\,,$

\begin{equation}\label{C4}
\begin{array}{l}
\displaystyle{
\Vert U(t)\Vert_\caH \leq M e^{-\omega t} \Big (\|U^0\|_\caH+\gamma\int_0^t   e^{\omega s}   \Vert U(s)\Vert_\caH \,ds }\\
\hspace{3 cm}\displaystyle{
+\alpha +\sum_{i=1}^I \gamma_i \int_{\tau_i}^t e^{\omega s} \Vert U(s-\tau_i )\Vert_\caH ds
\Big )\,,}
\end{array}
\end{equation}
where

\begin{equation}\label{C3}
\alpha =\sum_{i=1}^I\gamma_i \int_0^{\tau_i} e^{\omega s} \Vert U(s-\tau_i)\Vert_\caH ds\,.
\end{equation}
Thus, for $t\ge\tau,$ we obtain
\begin{equation}\label{C5}
\begin{array}{l}
\displaystyle{
\Vert U(t)\Vert_\caH \leq M e^{-\omega t} \Big (\|U_0\|_\caH+\gamma\int_0^t   e^{\omega s}   \Vert U(s)\Vert_\caH \,ds +\alpha }\\
\hspace{5.5 cm}\displaystyle{
+\sum_{i=1}^I \gamma_i\int_0^{t-\tau_i }
e^{\omega (s+\tau_i )}\Vert U(s)\Vert_\caH\,ds
\Big )}\\
\displaystyle{
\hspace {1.4 cm}\le  M e^{-\omega t} (
\Vert U_0\Vert_\caH +\alpha )
+\gamma  M e^{-\omega t} (1+e^{\omega \tau })
\int_0^t   e^{\omega s}   \Vert U(s)\Vert_\caH \,ds\,.}
\end{array}
\end{equation}
For $t\le [0,\tau ],$ one can estimate
\begin{equation}\label{P1}
\begin{array}{l}
\displaystyle{
\sum_{i=1}^I\gamma_i \int_0^t e^{\omega s}\Vert U(s-\tau_i )\Vert_\caH\, ds}\\
\hspace{2 cm}\displaystyle{
\le
\sum_{i=1}^I\gamma_i \int_0^{\tau_i} e^{\omega s}\Vert U(s-\tau_i )\Vert_\caH\, ds +
\sum_{i\ :\ \tau_i<t}\gamma_i \int_{\tau_i}^t e^{\omega s}\Vert U(s-\tau_i )\Vert_\caH\, ds}
\\
\hspace{2 cm}
\displaystyle{=\alpha + \sum_{i\ :\ \tau_i<t}\gamma_i \int_{0}^{t-\tau_i} e^{\omega (s+\tau_i)}\Vert U(s)\Vert_\caH\, ds}\\
\hspace{2 cm}
\displaystyle{\le\alpha + \sum_{i\ :\ \tau_i<t}\gamma_i e^{\omega\tau }\int_{0}^{t} e^{\omega s}\Vert U(s)\Vert_\caH\, ds}\\
\hspace{2 cm}
\displaystyle{\le\alpha +\gamma e^{\omega\tau }\int_{0}^{t} e^{\omega s}\Vert U(s)\Vert_\caH\, ds}\,.
\end{array}
\end{equation}
Therefore, from (\ref{C1}), also for $t\le\tau,$ it results
\begin{equation}\label{P2}
\Vert U(t)\Vert_\caH \leq
M e^{-\omega t} (
\Vert U^0\Vert_\caH +\alpha )
+\gamma  M e^{-\omega t} (1+e^{\omega \tau })
\int_0^t   e^{\omega s}   \Vert U(s)\Vert_\caH \,ds\,.
\end{equation}
Then, from (\ref{C5}) and (\ref{P2}),
\begin{equation}\label{P3}
e^{\omega t}\Vert U(t)\Vert_\caH \leq
M  (
\Vert U^0\Vert_\caH +\alpha )
+\gamma  M  (1+e^{\omega \tau })
\int_0^t   e^{\omega s}   \Vert U(s)\Vert_\caH \,ds\,, \quad \forall\ t\ge 0\,.
\end{equation}
Therefore,    Gronwall's lemma yields
\begin{equation}\label{C7}
\displaystyle{
e^{\omega t}\Vert U(t)\Vert_\caH\le Me^{M\gamma (1+
e^{\omega \tau })t } (\Vert U^0\Vert_\caH +\alpha )\,, \quad t\ge 0\,.}
\end{equation}

Estimate (\ref{C7}) can be rewritten as
\begin{equation}\label{C8}
\Vert U(t)\Vert_\caH\le Me^{-(\omega -M\gamma (1+
e^{\omega \tau }))t}
(\Vert U_0\Vert_\caH +\alpha )\,, \quad t\ge 0\,.
\end{equation}

Then, exponential decay is ensured under the condition

\begin{equation}\label{C9}
\omega >M\gamma (1 +e^{\omega \tau })\,.
\end{equation}

Under the assumption (\ref{suF}), inequality (\ref{C9}) is satisfied if and only if $\tau$ satisfies the smallness condition
(\ref{sutau}). 
Then, the statement is proved. \end{proof}

We can restate the previous theorem for the particular case of model (\ref{abstractJEE}):

\begin{Theorem}\label{JEEmodel}
Assume that the nonlinear term $G$ satisfies the Lipschitz condition
 \begin{equation}\label{lip2}
\exists \gamma >0 \quad \mbox{\rm such\ that}\quad \Vert G(U_1)-G(U_1^*)\Vert_\caH
\le \gamma \Vert U_1-U_1^*\Vert_\caH ,
\quad \forall\ U_1, U_1^*\in \caH\,.
\end{equation}
Let $M, \omega, \gamma$ as in
 $(\ref{assumpexpdecay})$ and $(\ref{lip2}).$
Moreover, let us assume that $G(0)=0$ and   $(\ref{suF})$ hold.
If the time delay $\tau$ satisfies
 the {\em smallness} condition
\begin{equation}\label{sutau2}
 \tau <\tau_0^\prime :=\frac 1 {\omega }\ln\frac 1 {k\Vert {\mathcal B}\Vert_{\mathcal H}} \Big (\frac {\omega } {M }-\gamma \Big )\,,
 \end{equation}
 then
 there exists $\omega^* >0$ such that
the solution $U\in C([0,+\infty), {\mathcal H})$  of problem $(\ref{abstractJEE}),$
with $U^0\in {\mathcal H}$ and $g\in C([0,\tau ]; \caH),$ satisfies
\begin{equation}\label{exponentiald2}
\|U(t)\|_\caH\le M e^{-\omega^* t} \Big(\|U_0\|_\caH+k \int_0^\tau e^{\omega s}\Vert g(s)\Vert_\caH\, ds \Big),
\quad \forall t\ge 0.
\end{equation}
\end{Theorem}

\begin{Remark} {\rm
 Note that Theorem \ref{stab2} is very general.
For instance, it furnishes stability results for previously studied models for wave type equations (cfr. \cite{ANP10, SCL12, MCSS}
and subsections \ref{esempioDampedWave}-\ref{esempioMemory} below) and Timoshenko models (cfr.
\cite{Said}) eventually with the addition of a nonlinear term depending on the not delayed state and on a finite number of delayed states. Also, it includes recent stability results for problems with
viscoelastic damping and time delay (cfr. \cite{guesmia, AlNP}).}
\end{Remark}

\begin{Remark}\label{compareJEE}{\rm
Note that the condition (\ref{sutau2}) is a bit less general with respect to the one obtained in \cite{JEE15}. Indeed, here we need
$$ k\Vert {\mathcal B}\Vert_{\mathcal H} e^{\omega\tau}+\gamma < \frac {\omega} M \,,$$
instead of the condition
$$ k\Vert {\mathcal B}\Vert_{\mathcal H}e^{\omega\tau}+\gamma < \frac { e^{\omega\tau}-1}{M\tau}$$
 assumed there.
On the other hand, the present approach allows to extend the class of the problems for which the stability result holds.
}\end{Remark}

We now present a couple of examples, among the many, illustrating our previous abstract result.

\subsection{The damped wave equation}\label{esempioDampedWave}

Let $\Omega$ be a bounded domain in $\RR^n$ with a boundary $\Gamma$ of class $C^2.$
Let $f_j:\RR\rightarrow \RR$ be globally Lipschitz continuous functions, $j=1,2,$ satisfying
$f_1(0)=f_2(0)=0.$
Let us consider the following semilinear damped wave equation:
$$
\begin{array}{l}
\displaystyle{
u_{tt}(x,t)-\Delta u(x,t)+a(x)u_t(x,t)=f_1(u(x,t))+f_2(u(x, t-\tau ) ) \quad \mbox{\rm in } \ \Omega\times (0, +\infty),}\label{DW1}\\
\displaystyle{u(x,t)=0\quad \mbox{\rm in }\ \Omega\times (0,+\infty )},\label{DW2}\\
\displaystyle{u(x,t-\tau)=u_0(x, t),\ \ u_t(x,t-\tau)=u_1(x,t)\quad \mbox{\rm in}\ \Omega\times (0,\tau]}\,,\label{DW3}
\end{array}
$$
where $\tau >0$ is the time delay and the damping coefficient $a\in L^\infty (\Omega )$
satisfies
\begin{equation}\label{Cri2}
a(x)\ge a_0>0, \quad \mbox{\rm a.e.}\ \ x\in\omega\,,
\end{equation}
for some nonempty open subset $\omega$ of $\Omega$ satisfying some control geometric properties (see e.g. \cite{BLR}). The initial datum $(u_0, u_1)$ is taken in $C([0,\tau], H^1_0(\Omega)\times L^2(\Omega)).$

Setting $U= (u, u_t)^T,$ this problem can be rewritten in the form $(\ref{abstract})$
with ${\mathcal H}=  H^1_0(\Omega)\times L^2(\Omega),$
$${\mathcal A}=
\left (
\begin{array}{l}
0 \quad \quad\
 1\\
\Delta\ \ -a
\end{array}
\right )$$
and
$F_1(U(t), U(t-\tau))=(0, f_1(u(t)) +f_2(u(t-\tau )))^T\,.$
It is well--known that ${\mathcal A}$ generates a strongly continuous semigroup which is exponentially stable (see e.g. \cite{zuazua, Komornikbook}), thus the assumptions on $f_1, f_2$ ensure that Proposition \ref{WellP} and Theorem \ref{stab2}  apply to this model giving a well-posedness result and an exponential decay estimate of the energy for small values of the time delay $\tau.$

\subsection{The wave equation with memory }\label{esempioMemory}
Let $\Omega\subset\RR^n$ be an open bounded set with a smooth boundary and let $f_j,j=1,2,$ as in the previous example.
Let us consider the following problem:

\begin{eqnarray}
& &u_{tt}(x,t) -\Delta u (x,t)+
\int_0^\infty \mu (s)\Delta u(x,t-s) ds\nonumber \\
& &\hspace{3,5 cm}
 =f_1(u(x,t))
+f_2( u(x,t-\tau) )\quad \mbox{\rm in}\ \Omega\times
(0,+\infty),\label{1.1d}\\
& &u (x,t) =0\quad \mbox{\rm on}\ \partial\Omega\times
(0,+\infty),\label{1.2d}\\
& &u(x,t)=u_0(x, t)\quad \hbox{\rm
in}\ \Omega\times (-\infty, 0], \label{1.3d}
\end{eqnarray}
where  the initial datum $u_0$ belongs to a suitable space, the constant $\tau >0$ is the time delay and
the  memory kernel $\mu :[0,+\infty)\rightarrow [0,+\infty)$
is a locally absolutely continuous function
satisfying

i) $\mu (0)=\mu_0>0;$

ii) $\int_0^{+\infty} \mu (t) dt=\tilde \mu <1;$

iii) $\mu^{\prime} (t)\le -\alpha \mu (t), \quad \mbox{for some}\ \ \alpha >0.$

As in \cite{Dafermos}, we denote
\begin{equation}\label{eta}
\eta^t(x,s):=u(x,t)-u(x,t-s).
\end{equation}
Then we can restate (\ref{1.1d})--(\ref{1.3d})
as

\begin{eqnarray}
& &u_{tt}(x,t)= (1-\tilde \mu)\Delta u (x,t)+
\int_0^\infty \mu (s)\Delta \eta^t(x,s) ds\nonumber\\
& &\hspace{4 cm}
=f_1(u(x,t))+ f_2(u(x,t-\tau))\quad \mbox{\rm in}\ \Omega\times
(0,+\infty)\label{e1d}\\
& & \eta_t^t(x,s)=-\eta^t_s(x,s)+u_t(x,t)\quad \mbox{\rm in}\ \Omega\times
(0,+\infty)\times (0,+\infty ),\label{e2d}\\
& &u (x,t) =0\quad \mbox{\rm on}\ \partial\Omega\times
(0,+\infty)\label{e3d}\\
& &\eta^t (x,s) =0\quad \mbox{\rm in}\ \partial\Omega\times
(0,+\infty), \ t\ge 0,\label{e4d}\\
& &u(x,0)=u_0(x)\quad \mbox{\rm and}\quad u_t(x,0)=u_1(x)\quad \hbox{\rm
in}\ \Omega,\label{e5d}\\
& & \eta^0(x,s)=\eta_0(x,s) \quad \mbox{\rm in}\ \Omega\times
(0,+\infty), \label{e6d}
\end{eqnarray}
where
\begin{equation}\label{datiinizd}
\begin{array}{l}
u_0(x)=u_0(x,0), \quad x\in\Omega,\\
u_1(x)=\frac {\partial u_0}{\partial t}(x,t)\vert_{t=0},\quad x\in\Omega,\\
\eta_0(x,s)=u_0(x,0)-u_0(x,-s),\quad x\in\Omega,\  s\in (0,+\infty).
\end{array}
\end{equation}

Let us denote
${U}:= (u,u_t,\eta^t)^T.$ Then we can rewrite    problem (\ref{e1d})--(\ref{e6d})
in the abstract form
\begin{equation}\label{abstractd}
\left\{
\begin{array}{l}
{U}_t(t)={\mathcal A} {U}(t)+F_1(U(t),U(t-\tau)),\\
{U}(0)=(u_0,u_1, \eta_0)^T,
\end{array}
\right.
\end{equation}
where the operator ${\mathcal A}$ is defined by
\begin{equation}\label{Operator}
{\mathcal A}\left (
\begin{array}{l}
u\\v\\w
\end{array}
\right )
:=\left (
\begin{array}{l}
v\\
(1-\tilde\mu)\Delta u+\int_0^{\infty}\mu (s)\Delta w(s)ds\\
-w_s+v
\end{array}
\right ),
\end{equation}
with domain

\begin{equation}\label{dominioOpd}
\begin{array}{l}
{\mathcal D}({\mathcal A}):=\left\{
\ (u,v,\eta )^T\in   H^1_0(\Omega)\times H^1_0(\Omega)
\times L^2_{\mu}((0,+\infty);H^1_0(\Omega))\, :\right.\\\medskip
\hspace{0.5 cm}\left.
(1-\tilde\mu)u+\int_0^\infty \mu (s)\eta (s) ds \in H^2(\Omega)\cap H^1_0(\Omega),\ \eta_s\in  L^2_{\mu}((0,+\infty);H^1_0(\Omega))
\right\},
\end{array}
\end{equation}
where
$L^2_{\mu}((0, \infty);H^1_0(\Omega ))$ is the Hilbert space
of $H^1_0-$ valued functions on $(0,+\infty),$
endowed with the inner product
$$\langle \varphi, \psi\rangle_{L^2_{\mu}((0, \infty);H^1_0(\Omega ))}=
\int_{\Omega}\left (\int_0^\infty \mu (s)\nabla \varphi (x,s)\nabla \psi (x,s) ds\right )dx.
$$

\noindent
Denote by ${\mathcal H}$
the Hilbert space

$${\mathcal H}=
H^1_0(\Omega)\times L^2(\Omega)\times L^2_{\mu}((0, \infty);H^1_0(\Omega )),$$
equipped
  with the inner product

\begin{equation}\label{innerd}
\begin{array}{l}
\left\langle
\left (
\begin{array}{l}
u\\
v\\
w
\end{array}
\right ),\left (
\begin{array}{l}
\tilde u\\
\tilde v\\
\tilde w
\end{array}
\right )
\right\rangle_{\mathcal H}
:= \displaystyle{
 (1-\tilde\mu )\int_\Omega \nabla u\nabla\tilde u dx + \int_\Omega v\tilde v dx +
\int_{\Omega} \int_0^\infty \mu (s)\nabla w\nabla\tilde w ds dx}.
\end{array}
\end{equation}
It is well--known (see e.g. \cite{Pata}) that the operator ${\mathcal A}$ generates an exponentially stable semigroup. Thus, Proposition \ref{WellP} and Theorem \ref{stab2} guarantee well--posedness and exponential stability, for small delays, also 
for the model  (\ref{1.1d})--(\ref{1.3d}).

\section{The case of small data with general nonlinearities}\label{secgeneralNL}

\setcounter{equation}{0}

Here we consider a more general class of nonlinearities but assume that
$\caA$ is a negative selfadjoint operator in $\caH$.
In this case, $\caA$  generates an analytic semi-group
(see \cite[Example IX.1.25]{kato67}) and existence results for problem  (\ref{abstract}) can be obtained
for    nonlinear terms   satisfying the next hypothesis (\ref{hypoF1}).
More precisely, we recall that $\caV=D((-\caA)^\frac{1}{2})$
is a Hilbert space with the norm
\[
\|U\|_\caV^2=((-\caA)^\frac{1}{2}U, (-\caA)^\frac{1}{2}U),\  \forall\  U\in \caV.
\]
Furthermore if $\lambda_1$ is the smallest eigenvalue of $-\caA$, we have
\be\label{Poincare}
\lambda_1 \|U\|_\caH^2\leq  \|U\|_\caV^2, \quad \forall\  U\in \caV.
\ee
Then we assume that
there exist a positive real number $\beta<\frac12$, a  constant $C_0$ and two continuous functions
$h_1,h_2$ from $[0,\infty)^4$ to $[0,\infty)$  such that, for all $i=1, \dots, I,$
\beq\label{hypoF1}
&&\Vert F_i(U_1, V_1)-F_i(U_2,V_2)\Vert_\caH \le C_0\Vert U_1-U_2\Vert_\caH
\\
\nonumber
&&\hspace{1 cm}+h_1(\Vert U_1\Vert_\caV, \Vert U_2\Vert_\caV,\Vert V_1\Vert_\caV,\Vert V_2\Vert_\caV)
 \Vert U_1 -U_2\Vert_\caV
\\
\nonumber
&&\hspace{1 cm}+h_2(\Vert U_1\Vert_\caV, \Vert U_2\Vert_\caV,\Vert V_1\Vert_\caV,\Vert V_2\Vert_\caV)
 \Vert V_1 -V_2\Vert_{D((-\caA)^\beta)}\,,
\eeq
for all $(U_1, U_2), (V_1, V_2) \in \caV\times\caV$.

The following local  existence result holds
(compare with Theorem 1 of \cite{Oliva:99}).


\begin{Proposition}\label{WellP*}
Assume that the functions  $F_i$ satisfy {\rm (\ref{hypoF1})}.
Then for any initial datum  $U_0\in C ([0,\tau];\caV)\cap C^{0,\theta}([0, \tau], D((-\caA)^
\beta))$,
with $\beta<\frac12$ from the assumption $(\ref{hypoF1})$ and $\theta=\min\{\beta, \frac{1}{2}-\beta\}$, there exist a time $T_\infty \in (0, +\infty ]$ and a unique  solution
$U\in C([0, T_\infty ), \caV)\cap C^1((0, T_\infty ), \caH )$  of problem $(\ref{abstract}).$
\end{Proposition}

\noindent \begin{proof}
Like in the proof of Proposition \ref{WellP}, we look at the problem
(\ref{abstract}) in the time interval $[0,\tau_{\rm min}) $ where, as before, $\tau_{\rm min}=\min_{i=1,\ldots, I} \tau_i$.
Then the model can be rewritten as
\begin{equation}\label{local1}
\left\{
\begin{array}{ll}
& U_{t}(t) = \caA U(t)+g(t, U(t))\quad \mbox{\rm in }
(0,\tau_{\rm min}),\\
& U(0)=U^0,
\end{array}
\right.
\end{equation}
where $g(t, U(t))=\sum_{i=1}^IF_i (U(t), U(t-\tau_i))$. Recall that the terms $U(t-\tau_i)$ can be regarded as known data for $t\in [0, \tau_{\rm min})\,.$
Then owing to (\ref{hypoF1}) the nonlinear part  satisfies
 \begin{equation}\label{hypPazy}
 \begin{array}{l}
 \displaystyle{
 \Vert g(t_1, U_1)-g(t_2, U_2)\Vert_{\caH}=
 \Big\Vert
 \sum_{i=1}^I[F_i (U_1, U_0(t_1-\tau_i+\tau ))-
 F_i (U_2, U_0(t_2-\tau_i+\tau ))]\Big
\Vert_{\caH}}\\
\hspace{0,5 cm}\displaystyle{
\le \sum_{i=1}^I \Vert
F_i (U_1, U_0(t_1-\tau_i+\tau ))-
 F_i (U_2, U_0(t_2-\tau_i+\tau ))
 \Vert_{\caH}}\\
\hspace{0,5 cm} \displaystyle {\le IC_0\Vert U_1-U_2\Vert_{\caH}}\\
\hspace{1 cm}+\displaystyle{
\sum_{i=1}^I h_1 (\Vert U_1\Vert_\caV ,  \Vert
U_2\Vert_\caV , \Vert U_0(t_1-\tau_i+\tau )\Vert_\caV , \Vert U_0(t_2-\tau_i+\tau )\Vert_\caV  )\Vert U_1-U_2\Vert_\caV\,}\\
\hspace{0,5 cm} \displaystyle {+
\sum_{i=1}^I h_2 (\Vert U_1\Vert_\caV ,  \Vert U_2\Vert_\caV , \Vert U_0(t_1-\tau_i+\tau )\Vert_\caV , \Vert U_0(t_2-\tau_i+\tau )\Vert_\caV  )\cdot }
\\
\hspace{5 cm}
\displaystyle{\cdot
\Vert U_0(t_1-\tau_i+\tau )-U_0 (t_2-\tau_i+\tau )\Vert_{D(-A^\beta )}\,.}
\end{array}
 \end{equation}
 From (\ref{hypPazy}), using (\ref{Poincare}), easily follows
$$ \Vert g(t_1, U_1)-g(t_2, U_2)\Vert_{\caH}\le
L(R)(\vert t_1-t_2\vert^\theta +\Vert U_1-U_2\Vert_\caV  )\,,$$
for all $(U_1, U_2)\in\caV\times\caV$ with
$\Vert U_1\Vert_\caV , \Vert U_2\Vert_\caV \le R$
and for all $t_1, t_2\in [0, \tau_{\rm min}]\,.$
Therefore, we can apply Theorem 6.3.1 of \cite{pazy} and deduce the existence of a unique local solution $U\in C([0,\overline t), \caV)\cap C^1((0,\overline t), \caH )$ defined in a time interval  $[0,\overline t)$ with $\overline t\le\tau_{\rm min}\,.$
If $\overline t=\tau_{\rm min},$ and $\Vert U(\tau_{\rm min})\Vert_\caV <+\infty, $ then one can extend the solution $U$ for times $t>\tau_{\rm min},$ by considering on the time interval $(\tau_{\rm min}, 2\tau_{\rm min})$ the problem
\begin{equation}\label{local2}
\left\{
\begin{array}{ll}
& U_{t}(t) = \caA U(t)+g_1(t, U(t))\quad \mbox{\rm in }
(\tau_{\rm min},2\tau_{\rm min})\\
& U(0)=U(\tau_{\rm min}^-),
\end{array}
\right.
\end{equation}
where $g_1(t, U(t))=\sum_{i=1}^IF_i (U(t), U(t-\tau_i))$. Note that, since we know the solution $U(t)$ for $t\in [0,\tau_{\rm min}]$ from the first step, the terms $U(t-\tau_i)$ can be regarded as known data for $t\in [\tau_{\rm min},2\tau_{\rm min})\,.$

Observe also that
\be\label{holdercont}
U\in C^{0,\theta }([0,\tau_{\rm min}], D(-\caA)^\beta )).
\ee
Indeed from the proof of
Theorem 6.3.1 of
\cite{pazy}, we see that
\[
U(t)=(-\caA)^{-\frac{1}{2}}y(t),
\]
with $y\in C([0,\tau_{\rm min}], \caH)$
given by
\[
y(t)=S(t)(-\caA)^{\frac{1}{2}}U^0+\int_{0}^t (-\caA)^{\frac{1}{2}} S(t-s)g(s, (-\caA)^{-\frac{1}{2}}y(s))\,ds,
\forall t\in [0,\tau_{\rm min}].
\]
Hence for any $t\in [0,\tau_{\rm min})$ and $h>0$ such that
$t+h\leq \tau_{\rm min}$, we have
\[
y(t+h)-y(t)=
(S(h)-I)S(t)(-\caA)^{\frac{1}{2}}U^0+R_2(t,h)+R_3(t,h),
\]
where
\beqs
R_2(t,h)&=&
\int_{0}^t (S(h)-I)(-\caA)^{\frac{1}{2}} S(t-s)g(s, (-\caA)^{-\frac{1}{2}}y(s))\,ds
\\
R_3(t,h)&=&\int_t^{t+h}  (-\caA)^{\frac{1}{2}} S(t+h-s)g(s, (-\caA)^{-\frac{1}{2}}y(s))\,ds.
\eeqs
Accordingly, we have
\beqs
(-\caA)^{\beta}(U(t+h)-U(t))
&=&(-\caA)^{\beta-\frac{1}{2}}(y(t+h)-y(t))
\\
&=&
(S(h)-I) (-\caA)^{\beta} S(t)U^0+(-\caA)^{\beta-\frac{1}{2}}(R_2(t,h)+R_3(t,h)),
\eeqs
and since $(-\caA)^{\beta-\frac{1}{2}}$ is a bounded operator from $\caH$ into itself (see
\cite[Lemma 2.6.3]{pazy}), one gets
\be\label{estholder1}
\|(-\caA)^{\beta}(U(t+h)-U(t))\|_\caH
\lesssim \|(S(h)-I) (-\caA)^{\beta} S(t)U^0\|_\caH+\|R_2(t,h)\|_\caH+\|R_3(t,h)\|_\caH.
\ee
The estimates (6.3.15) and (6.3.16) of \cite{pazy} give
\be \label{estholderremainder}
\|R_2(t,h)\|_\caH+\|R_3(t,h)\|_\caH\lesssim h^\beta\lesssim h^\theta,
\ee
hence to get the H\"older continuity (\ref{holdercont}), it remains to show that
\be \label{estholder2}
\|(S(h)-I) (-\caA)^{\beta} S(t)U^0\|_\caH\lesssim h^\theta.
\ee
But we first notice that the estimates (2.6.26)  of \cite{pazy} yields
\[
\|(S(h)-I) (-\caA)^{\beta} S(t)U^0\|_\caH\lesssim h^{\frac{1}{2}-\beta}
\| (-\caA)^{\frac{1}{2}} S(t)U^0\|_\caH,
\]
and since $(-\caA)^{\frac{1}{2}} S(t)=S(t)(-\caA)^{\frac{1}{2}}$
and the semigroup is of contractions, we get
\[
\|(S(h)-I) (-\caA)^{\beta} S(t)U^0\|_\caH\lesssim h^{\frac{1}{2}-\beta}
\| (-\caA)^{\frac{1}{2}}U^0\|_\caH\lesssim h^{\frac{1}{2}-\beta} \|U^0\|_\caV,
\]
which proves (\ref{estholder2}).

Then, as before, one can apply Theorem 6.3.1 of \cite{pazy}
extending the previously found solution. One can eventually iterate this procedure by obtaining a solution
$$U\in C([0, T_\infty ), \caV)\cap C^1((0, T_\infty ), \caH )$$
  of problem $(\ref{abstract}),$  satisfying  $\lim_{t\rightarrow T_\infty ^-}\Vert U(t)\Vert_\caV =+\infty $
if $T_\infty <+\infty\,.$
\end{proof}

We now give an exponential stability result   for {\sl small} initial data. Note that we do not require here any restriction on the size of the time delays.
For that purpose, we need the additional assumption on our nonlinear functions, namely
we suppose that there exist a    positive  constant $C_1$ and a  continuous function
$h_3$ from $[0,\infty)^{I+1}$ to $[0,\infty)$  satisfying $h_3(0)=0$
and  such that
\be\label{H10bis}
 \vert \sum_{i=1}^I(W,F_i(U, V_i))_\caH\vert \le \|W\|_\caH \left(C_1\Vert U\Vert_\caH
+h_3(\Vert U\Vert_\caV,  \Vert V_1\Vert_\caV,\cdots,\Vert V_I\Vert_\caV)
 \Vert U \Vert_\caV\right),
 \ee
 for all $W\in \caH$, $U, V_i\in \caV$, $i=1,\cdots, I$.

\begin{Theorem}\label{Wellnonlineargeneral}
Assume that $(\ref{hypoF1})$ and $(\ref{H10bis})$ are satisfied.
With $C_1$ from the assumption $(\ref{H10bis}),$
we assume that
\be\label{C1general}
\frac{C_1}{\lambda_1}-1<0.
\ee
Then there exist $K_0>0$ small enough and $\gamma_0<1$ (depending on $K_0$) such that for all
$K\in (0, K_0]$ and   $U_0\in C([0,\tau];\caV )$ satisfying
\begin{equation}\label{assump}
\|U_0(t)\|_\caV<\gamma_0 K, \forall \ t\in [-\tau,0],
\end{equation}
problem $(\ref{abstract})$ has a global solution  $U$ that satisfies the exponential decay estimate
 \begin{equation}\label{exponentialdbis}
\|U(t)\|_\caH\le  M e^{-\tilde \omega t}
\quad \forall t\ge 0\,,
\end{equation}
for a   positive constant  $M$ depending on $U_0$ and a suitable  positive constant $\tilde\omega$.
 \end{Theorem}
\begin{proof}
 By Proposition \ref{WellP*}, there exists $T_\infty>0$ such that   problem (\ref{abstract})  has a unique solution
  $U\in C([0, T_\infty), \caV)\cap C^1((0,T_\infty), \caH)$.
As in \cite{Liu:02}, for $K\in (0, K_0]$ with $K_0$ fixed later on, we look at
\be\label{defT0}
T_0=\sup\{\delta\in (0,\infty): \|U(t)\|_\caV\leq K, \forall t\in (0,\delta)\}.
\ee
Our assumption (\ref{assump}) clearly guarantees that $T_0>0$. We will now show that $T_0=+\infty$
by a contradiction argument. Indeed if we assume that $T_0$ is finite, then
by its definition, we will have
\be\label{16}
\|U(t)\|_\caV<K, \forall t\in [-\tau, T_0),
\ee
and
\be\label{17}
\|U(T_0)\|_\caV=K.
\ee
In a first step, for $t\in [0, T_0)$, we estimate
$\frac{d}{dt} \|U(t)\|_\caH^2$.
Indeed by  (\ref{abstract}), for all  $t\in [0, T_0)$,  we have
\beqs
\frac{d}{dt} \|U(t)\|_\caH^2
&=&2(U(t), U_t(t))_\caH
\\
&=&2(U(t), \caA U(t)+\sum_{i=1}^I F_i(U(t), U(t-\tau_i)) )_\caH
\\
&=&-2\|U(t)\|^2_\caV+2\sum_{i=1}^I (U(t), F_i(U(t), U(t-\tau_i)) )_\caH\,.
\eeqs
Hence using the assumption (\ref{H10bis}), we get
\beqs
\frac{d}{dt} \|U(t)\|_\caH^2
&\leq& -2\|U(t)\|^2_\caV+2 C_1\Vert U(t)\Vert_\caH^2
\\
\nonumber
&+&2h_3(\Vert U(t)\Vert_\caV,  \Vert U(t-\tau_1)\Vert_\caV,\cdots,\Vert U(t-\tau_I)\Vert_\caV)
\Vert U(t)\Vert_\caH \Vert U(t) \Vert_\caV.
\eeqs
Therefore by (\ref{16}), we deduce that
\beqs
\frac{d}{dt} \|U(t)\|_\caH^2
&\leq& -2\|U(t)\|^2_\caV+2 C_1\Vert U(t)\Vert_\caH^2
\\
\nonumber
&+&2C_4(K)
\Vert U(t)\Vert_\caH \Vert U(t) \Vert_\caV,
\eeqs
where $C_4(K)=\max_{0\leq y, z_i\leq K} h_3(y,z_1, \cdots, z_I)$ is a constant that depends on $K$
and is a non-decreasing function of $K$. Using (\ref{Poincare}), we arrive at

\beqs
\frac{d}{dt} \|U(t)\|_\caH^2
&\leq& 2\left(-1+\frac{C_1}{\lambda_1} +\frac{C_4(K)}{\sqrt{\lambda_1}}\right)
\Vert U(t)\Vert_\caV^2.
\nonumber
\eeqs

From our assumption (\ref{C1general}), we can choose $K_0$
small enough such that
\[
-\omega=\left(-1+\frac{C_1}{\lambda_1} +C_4(K_0)\right)
\]
is negative. Hence
for all $K\in (0,K_0]$, the previous estimate implies that
\be\label{22}
\frac{d}{dt} \|U(t)\|_\caH^2
\leq -2\omega
\Vert U(t)\Vert_\caV^2\leq -2\omega\lambda_1
\Vert U(t)\Vert_\caH^2, \ \forall\  t\in [0,T_0).
\ee
This estimate obviously implies that
 \be\label{23}
\|U(t)\|_\caH^2
\leq  e^{-2\omega\lambda_1t}
\Vert U(0)\Vert_\caH^2,\  \forall \ t\in [0,T_0).
\ee
But the first estimate of (\ref{22}) means that
\[
\frac{d}{dt} \|U(t)\|_\caH^2 +2\omega
\Vert U(t)\Vert_\caV^2\leq 0,
\]
and if we multiply this estimate by $e^{\omega't}$ for some $\omega'\in (0,\omega]$ fixed later on, we find that
\[
\frac{d}{dt} (e^{\omega't}\|U(t)\|_\caH^2) +2\omega'
e^{\omega't}\Vert U(t)\Vert_\caV^2\leq  \omega' e^{\omega't}\|U(t)\|_\caH^2,
\]
and using (\ref{23}), we get
\be\label{25}
\frac{d}{dt} (e^{\omega't}\|U(t)\|_\caH^2) +2\omega'
e^{\omega't}\Vert U(t)\Vert_\caV^2\leq  \omega e^{(\omega'-2\lambda_1 \omega)t}\|U(0)\|_\caH^2.
\ee
Hence   we fix $\omega'\in (0,\omega]$ such that $\omega'<2\lambda_1 \omega$ (possible since $\lambda_1$ and $\omega$
are positive), and then integrate (\ref{25}) between 0 and $T_0$
to find
\begin{eqnarray*}
e^{\omega'T_0}\|U(T_0)\|_\caH^2 +2\omega'
\int_0^{T_0}e^{\omega't}\Vert U(t)\Vert_\caV^2\,dt
&\leq& \left(1+\frac{\omega}{2\lambda_1 \omega-\omega'}  (1-e^{(\omega'-2\lambda_1 \omega)T_0})\right)
\|U(0)\|_\caH^2\\
&\leq&\left(1+\frac{\omega}{2\lambda_1 \omega-\omega'}\right)
\|U(0)\|_\caH^2.
\end{eqnarray*}
This clearly implies that
there exists a positive constant $C(\omega, \lambda_1)$ depending only on
$\omega$ and $\lambda_1$ such that
\be\label{27}
\int_0^{T_0}e^{\omega't}\Vert U(t)\Vert_\caV^2\,dt
\leq C(\omega, \lambda_1)
\|U(0)\|_\caV^2.
\ee

We now estimate
$\frac{d}{dt} \|U(t)\|_\caV^2$.
First  we notice that
\beqs
\frac{d}{dt} \|U(t)\|_\caV^2&=&2((-\caA)^\frac{1}{2}U, (-\caA)^\frac{1}{2}U_t)_\caH
\\
&=&-2(\caA U,U_t)_\caH
\\
&=&-2(U_{t}(t) -\sum_{i=1}^I F_i(U(t), U(t-\tau_i)),U_t)_\caH
\\
&=&-2\|U_{t}(t)\|_\caH^2 +2\sum_{i=1}^I (F_i(U(t), U(t-\tau_i)),U_t)_\caH\,.
\eeqs
By  the assumption (\ref{H10bis}), we get
\[
\frac{d}{dt} \|U(t)\|_\caV^2\leq
-2\|U_{t}(t)\|_\caH^2 +2\|U_{t}(t)\|_\caH
(C_1 \Vert U(t)\Vert_\caH +C_4(K)
 \Vert U(t)\Vert_\caV).
\]
Young's inequality and (\ref{Poincare}) lead to
\be\label{IBbis}
\frac{d}{dt} \|U(t)\|_\caV^2\leq
(\frac{C_1}{\sqrt{\lambda_1}} +C_4(K))^2
 \Vert U(t)\Vert_\caV^2.
\ee
Now we proceed as in the proof of Lemma 2.1 of \cite{Liu:02}, namely
we multiply  this estimate by $e^{\omega't}$ to get
\[
\frac{d}{dt} (e^{\omega't} \|U(t)\|_\caV^2)\leq
(\omega'+(\frac{C_1}{\sqrt{\lambda_1}} +C_4(K))^2)
e^{\omega't} \Vert U(t)\Vert_\caV^2.
\]
Integrating this estimate between 0 and $t\in (0,T_0]$, one obtains
\[
e^{\omega't} \|U(t)\|_\caV^2\leq \|U(0)\|_\caV^2+(\omega'+(\frac{C_1}{\sqrt{\lambda_1}} +C_4(K))^2)
\int_0^te^{\omega's} \Vert U(s)\Vert_\caV^2\,ds,
\]
and therefore owing to (\ref{27}), one finally finds
\be\label{199}
\|U(t)\|_\caV^2\leq \left(1+(\omega'+(\frac{C_1}{\sqrt{\lambda_1}} +C_4(K))^2)C(\omega, \lambda_1)\right)
\|U(0)\|_\caV^2 e^{-\omega't},\  \forall  t\in (0,T_0].
\ee
Therefore if we  define $\gamma_0>0$ by
\[
\gamma_0^{-2}= 4\left(1+(\omega'+(\frac{C_1}{\sqrt{\lambda_1}} +C_4(K_0))^2)C(\omega, \lambda_1)\right),
\]
we clearly see that $\gamma_0<1$. Furthermore for any $U_0$ satisfying
(\ref{assump}), and reminding that $0<K\leq K_0$, we have
$$
\begin{array}{l}
\displaystyle{
\left(1+(\omega'+(\frac{C_1}{\sqrt{\lambda_1}} +C_4(K))^2)C(\omega, \lambda_1)\right)
\|U(0)\|_\caV^2}\\
\hspace{1,5 cm}\displaystyle{
\leq
\left((1+(\omega'+(\frac{C_1}{\sqrt{\lambda_1}} +C_4(K_0))^2)C(\omega, \lambda_1)\right)  \gamma_0^{2} K^2
=\frac{1}{4}K^2.
}
\end{array}
$$
Consequently, the estimate (\ref{199})
guarantees that
 \be\label{203}
\|U(t)\|_\caV\leq \frac{K}{2} e^{-\frac{\omega't}{2}}, \forall  t\in (0,T_0].
\ee
In particular, on gets
\[
\|U(T_0)\|_\caV\leq \frac{K}{2},
\]
that   clearly contradicts (\ref{17}). This means that $T_0$ is infinite and
we conclude by the estimate (\ref{203}).
\end{proof}

\section{Examples with general nonlinearities}\label{Examples}

\setcounter{equation}{0}

In the whole section, $\Omega$ denotes  an arbitrary bounded domain of $\mathbb{R}^d$,  $d\geq 1$,
with a Lipschitz boundary $\Gamma$.

\subsection{Delay-diffusion equations}\label{esempio1}

We consider the semilinear diffusion equation with time delay
 \begin{eqnarray}
& &u_{t}(t) -\Delta u (t)= f(u(t), u(t-\tau ))\quad \mbox{\rm in}\ \Omega\times (0,+\infty),\label{C1.1}\\
& & Bu(x,t)=0\quad \mbox{\rm on}\ \Gamma\times (0,+\infty),\label{C1.2}\\
& &u(x,t))=u_0(x,t)\quad \mbox{\rm in}\ \Omega\times [-\tau ,0],\label{C1.3}
\end{eqnarray}
where the constant  $\tau >0$ is the time delay and the initial datum $u_0$ belongs to the space $C([0,\tau];L^2(\Omega ))$. The operator $B$ is in the form
\[
Bu=\alpha \partial_n u+\alpha^\prime u,
\]
with either $\alpha=0$ and $\alpha^\prime =1$ corresponding to the case of Dirichlet boundary conditions
or $\alpha=1$ and $\alpha^\prime \geq 0$ (with $\alpha^\prime\in L^\infty(\partial\Omega)$)
 corresponding to the case of Neumann--Robin boundary conditions.

Analogous problems have been considered by Friesecke  in \cite{Friesecke}, where an exponential stability result is obtained for small time delay, under some growth conditions on the locally Lipschitz function $f\,,$ by using a Lyapunov functional approach. We may also quote  the paper of Pao \cite{Pao:96} where a coupled system of parabolic semilinear equations with delays is studied. Using the method of upper and lower solutions to investigate existence and asymptotic behavior, sufficient conditions for stability and instability are given, under some monotonicity properties on the reaction term $f\,,$ independent of the time delays.
 In the same spirit, in \cite{Freedman-Zhao:97} the global asymptotic behavior of some quasimonotone reaction--diffusion system with delays is analyzed. A trichotomy of the global dynamics is established via linearization.

For further uses, in the case of Neumann conditions ($\alpha=1, \alpha^\prime =0$), we fix
  a  positive real parameter  $\varepsilon$ (that may depend on $f$), otherwise we set  $\varepsilon=0$.
Now we assume that the nonlinearity $f:\RR^2\rightarrow \RR$ satisfies the following assumption:
there exist a non-negative constant $\alpha_0$ (that may depend on $\varepsilon$)  and two polynomials
$P_1$ and $P_2$ (of one real variable) of degree $n_1$ and $n_2$
in the form
\[
P_i(X)=\sum_{j=1}^{n_i} \alpha_{i,j} X^j,
\]
with non negative real numbers $\alpha_{i,j}$ such that
\beq
\nonumber
&&|f(x_1,y_1)-f(x_2,y_2)+\varepsilon(x_1-x_2)|\leq \alpha_0 |x_1-x_2|+P_1(|x_1|+|y_1|+|x_2|+|y_2|)  |x_1-x_2|
\\
&&\hspace{2cm}+P_2(|x_1|+|y_1|+|x_2|+|y_2|)  |y_1-y_2|, \forall (x_1,y_1),(x_2,y_2)\in \RR^2.
\label{hypoF1b}
\eeq
In particular this assumption means that $f$ is
only locally Lipschitz.

Under this assumption, let us now show that problem (\ref{C1.1})--(\ref{C1.3}) enters in the framework of section \ref{secgeneralNL}. Indeed in such a situation, we take    $\caH =L^2(\Omega )$
and define $\caA$ as follows:
\[
D(\caA):=\{u\in H^1(\Omega): \Delta u\in L^2(\Omega) \hbox{ and  satisfying } Bu= 0 \hbox{ on } \Gamma\},
\]
and
\[
\caA u=\Delta u-\varepsilon u,\ \  \forall\  u\in D(\caA).
\]
Note that by Lemmas 1.5.3.7 and  1.5.3.9 of \cite{grisvard:85a},
  for any $u\in \{u\in H^1(\Omega): \Delta u\in L^2(\Omega)\}$, $Bu$ has a meaning (as element of $H^{-\frac{1}{2}}(\Gamma)$
  if $\alpha>0$). It is not difficult to show that
  $-\caA$ is a positive selfadjoint operator in $\caH$ since it is the Friedrichs extension of the triple
  $(\caH, \caV, a)$ where
 $\caV= H^1_0(\Omega)$ in case of Dirichlet boundary conditions otherwise
 $\caV= H^1(\Omega)$, and
 the sesquilinear form $a$ is defined by
 \[
 a(u,v)=\int_\Omega (\nabla u\cdot \nabla \bar v +\varepsilon u\cdot  \bar v)\,dx
 +\int_\Gamma \alpha' u\cdot  \bar v\,d\sigma(x),
 \]
 that is symmetric, continuous and strongly coercive on $\caV$. Before going on, let us notice that
 in the case of Neumann boundary conditions, the smallest eigenvalue $\lambda_1$ of $-\caA$ is $\varepsilon$,
otherwise  it does not depend on $\varepsilon$
and corresponds to the smallest eigenvalue of the Dirichlet problem
or to the Robin eigenvalue problem.

By introducing the function
\[
F_1(x,y)=f(x,y)+\varepsilon x,
\]
we see that problem (\ref{C1.1})--(\ref{C1.3}) can be written as (\ref{abstract}) with $I=1$, $U(t)=u(\cdot, t)$
and $U_0(t)=u_0(\cdot, t)$.
Consequently the next local existence result follows from Proposition \ref{WellP*}.

\begin{Proposition}\label{WellP*ex1}
Assume that $(\ref{hypoF1b})$ holds and that $d$ satisfies
\begin{equation}\label{cond1}
d\leq 2\Big (1+\frac{1}{n_1}\Big )
\quad
\hbox{ and }\quad  d< 2\Big (1+\frac{1}{n_2}\Big ).
\end{equation}
Then there exists $\beta\in (0,\frac12)$
such that for any initial datum
$$u_0\in C ([0,\tau];\caV)\cap C^{0,\theta}([0, \tau], D((-\caA)^
\beta)),$$  with $\theta=\min\{\beta, \frac{1}{2}-\beta\}$,  there exist a time
$T_\infty \in (0, +\infty ]$ and a unique  solution
$u\in C([0, T_\infty ), \caV)\cap C^1((0, T_\infty ), \caH)$  of problem $(\ref{C1.1})-(\ref{C1.3}).$
\end{Proposition}
\begin{proof}
It suffices to check that $F_1$ defined above satisfies (\ref{hypoF1}). For that purpose,
we see that (\ref{hypoF1b}) implies that (for shortness we write $F$ instead of $F_1$)
\beqs
\nonumber
&&|F(x_1,y_1)-F(x_2,y_2)|\leq \alpha_0  |x_1-x_2|+P_1(|x_1|+|y_1|+|x_2|+|y_2|)  |x_1-x_2|
\\
&&\hspace{2cm}+P_2(|x_1|+|y_1|+|x_2|+|y_2|)  |y_1-y_2|, \forall (x_1,y_1),(x_2,y_2)\in \RR^2.
\eeqs
and consequently for any $u_i, v_i\in \caV$,
\beqs
\nonumber
&&\|F(u_1,v_1)-F(u_2,v_2)\|_\caH\leq \alpha_0 \|u_1-u_2\|_\caH+\|P_1(|u_1|+|v_1|+|u_2|+|v_2|)  |u_1-u_2|\|_\caH
\\
&&\hspace{2cm}
+\|P_2(|u_1|+|v_1|+|u_2|+|v_2|)  |v_1-v_2|\|_\caH.
\eeqs
Hence by the form of $P_i$
and H\"older's inequality, we arrive at
\beqs
\nonumber
\|F(u_1,v_1)-F(u_2,v_2)\|_\caH&\leq& \alpha_0 \|u_1-u_2\|_\caH\\
&+&C\sum_{j=1}^{n_1}\|(|u_1|+|v_1|+|u_2|+|v_2|)^j\|_{L^{q_1}(\Omega)}  \|u_1-u_2|\|_{L^{p_1}(\Omega)}
\\
&+&
C\sum_{j=1}^{n_2}\|(|u_1|+|v_1|+|u_2|+|v_2|)^j\|_{L^{q_2}(\Omega)} \|v_1-v_2\|_{L^{p_2}(\Omega)},
\eeqs
for some $C>0$ and $p_i, q_i>2$ chosen below and such that
\[
\frac{1}{p_i}+\frac{1}{q_i}=\frac{1}{2},\ \  i=1,2.
\]
Now as $D((-\caA)^{\frac12}) \hookrightarrow H^1(\Omega)$ \footnote{by $X\hookrightarrow Y$, we mean that $X$   is continuously embedded into $Y$}
and $D((-\caA)^{0})=L^2(\Omega)$, by interpolation
we get that
\[
D((-\caA)^
\beta)\hookrightarrow  H^{2\beta}(\Omega), \ \ \forall\  \beta \in [0,\frac12].
\]
Then
we notice that the Sobolev embedding theorem guarantees that
$H^1(\Omega)\hookrightarrow L^{p_1}(\Omega)$ and $H^{2\beta}(\Omega)\hookrightarrow L^{p_2}(\Omega)$
as soon as
\[
1-\frac{d}{2}\geq -\frac{d}{p_1} \hbox{ and } 2\beta-\frac{d}{2}\geq -\frac{d}{p_2}.
\]
Since $\beta$ can be taken as close to $\frac12$ as we want, we get equivalently
\be\label{sn:cond1}
1-\frac{d}{2}\geq -\frac{d}{p_1} \hbox{ and } 1-\frac{d}{2}> -\frac{d}{p_2}.
\ee
In a second step we now need to estimate terms like
\[
\|(|u_1|+|v_1|+|u_2|+|v_2|)^j\|_{L^{q_i}(\Omega)}
\]
for all $1\leq j\leq n_i$. But clearly we have
\[
\|(|u_1|+|v_1|+|u_2|+|v_2|)^j\|_{L^{q_i}(\Omega)}\leq C_2  (\|u_1\|_{L^{jq_i}(\Omega)}^j+\|v_1\|_{L^{jq_i}(\Omega)}^j
+\|u_2\|_{L^{jq_i}(\Omega)}^j+\|v_2\|_{L^{jq_i}(\Omega)}^j),
\]
for some $C_2>0$ (that depends on $j$). Therefore our last assumptions are
\[
1-\frac{d}{2}\geq -\frac{d}{jq_i}, \ \  \forall \ 1\leq j\leq n_i,
\]
or equivalently (since $p_i$ and $q_i$ are conjugates)
\be\label{sn:cond2}
1-\frac{d}{2}\geq \frac{d}{j} (\frac{1}{p_i}-\frac12),\ \   \forall\  1\leq j\leq n_i.
\ee
This condition and (\ref{sn:cond1}) guarantee that
\beqs
\nonumber
\|F(u_1,v_1)-F(u_2,v_2)\|_\caH&\leq& \alpha_0  \|u_1-u_2\|_\caH+C\sum_{j=1}^{n_1}(\|u_1\|_{\caV}^j+\|v_1\|_{\caV}^j
+\|u_2\|_{\caV}^j)  \|u_1-u_2|\|_{\caV}
\\
&+&
C\sum_{j=1}^{n_2}(\|u_1\|_{\caV}^j+\|v_1\|_{\caV}^j
+\|u_2\|_{\caV}^j) \|v_1-v_2\|_{D((-\caA)^\beta ) },
\eeqs
for some $C>0$ and some $\beta\in (0,\frac12)$, which yields (\ref{hypoF1}) where $C_0=
\alpha_0$ and $h_i$
are polynomials with positive coefficients.

Finally it is an easy exercise to check that
conditions (\ref{cond1}) are equivalent to the existence
of $p_1>2$ and  $p_2>2$ satisfying (\ref{sn:cond1}) and (\ref{sn:cond2}).
\end{proof}

Concerning global existence and exponential decay, owing to Theorem \ref{Wellnonlineargeneral}, we can state the next result.

\begin{Theorem}\label{Wellnonlineargeneral:ex1}
Assume that conditions $(\ref{hypoF1b}), (\ref{cond1})$ as well as
\be\label{cond3}
f(0,y)=0, \ \ \forall\  y\in \RR,
\ee
are satisfied.
With $\alpha_0$ from the assumption $(\ref{hypoF1b}),$
we assume that
\be\label{C1general:ex1Neu}
\alpha_0<\varepsilon,
\ee
in the case of Neumann boundary conditions, and
\be\label{C1general:ex1}
\alpha_0<\lambda_1,
\ee
otherwise.
Then there exist $K_0>0$ small enough and $\gamma_0<1$ (depending on $K_0$) such that for all
$K\in (0, K_0]$ and   $u_0\in C([0,\tau];\caV )$ satisfying
\begin{equation}\label{assump:ex1}
\|u_0(t)\|_\caV<\gamma_0 K, \ \ \forall\  t\in [-\tau,0],
\end{equation}
problem $(\ref{C1.1})-(\ref{C1.3})$ has a global solution  $u$ that satisfies the exponential decay estimate
 \begin{equation}\label{exponentialdbis:ex1}
\|u(t)\|_\caH\le  M e^{-\tilde \omega t}
\quad \forall \ t\ge 0\,,
\end{equation}
for a   positive constant  $M$ depending on $u_0$ and a suitable  positive constant $\tilde\omega$.
 \end{Theorem}
\begin{proof}
It suffices to check that (\ref{H10bis}) and (\ref{C1general}) hold.
The first condition directly follows from Cauchy-Schwarz's inequality and the assumptions (\ref{cond3}) and (\ref{hypoF1b})
where $C_1=\alpha_0.$ Then the second condition (\ref{C1general})
simply becomes either
(\ref{C1general:ex1Neu}) in the Neumann case or (\ref{C1general:ex1}) in the two other cases.
\end{proof}

Our general setting covers a very large number of concrete examples.
Let us mention the following cases:
\\ 1. {\bf Diffusive logistic equations with delay}.
In that case, $f$ is given by
\[
f(x,y)=a x-b x^2+c xy,
\]
with $a,b,c\in L^\infty(\Omega)$. This example covers  the Hutchinson equation by taking $a=\alpha\in \RR$,
$c=-\alpha$ and $b=0$.
 In such a case, the condition (\ref{hypoF1b}) holds with
$\alpha_0=\sup_\Omega |a+\varepsilon|$, and $n_1=n_2=1$.
Hence local existence (i.e. Theorem \ref{WellP*ex1}) holds for any $d\leq 3$, while
exponentiel decay for sufficiently small initial data (i.e. Theorem \ref{Wellnonlineargeneral:ex1}) holds under the additional assumption
that
\be\label{condaNeu}
\sup_\Omega a<0,
\ee
in the case of Neumann boundary conditions (by chosing $\varepsilon>0$ large enough), and
\be\label{condaDir}
\sup_\Omega |a| <\lambda_1,
\ee
in the other cases.
\\ 2. {\bf The modified Hutchinson equation}.
In that case, $f$ is given by (see \cite{Memory:91,Friesecke})
\[
f(x,y)=\alpha x (1+\beta y+\gamma y^2+\delta y^3),
\]
with $\alpha, \beta, \gamma, \delta\in \RR$.
 In such a case, the condition (\ref{hypoF1b}) holds with
$\alpha_0=|\alpha+\varepsilon|$, and $n_1=n_2=3$.
Hence local existence (i.e. Theorem \ref{WellP*ex1}) holds for any $d\leq 2$, while
exponentiel decay for sufficiently small initial data (i.e. Theorem \ref{Wellnonlineargeneral:ex1}) holds under the additional assumption
$\alpha<0$
in the case of Neumann boundary conditions, and
$-\lambda_1<\alpha<\lambda_1$
in the other cases.
\\ 2. {\bf A cubic nonlinearity}.
The case where  $f$ is given by
\[
f(x,y)=-x^2 y
\]
considered in \cite{Inoueetall:77,Lighbourne:81,Friesecke} is also covered by our setting
in the case of Dirichlet or Robin boundary conditions,
since (\ref{hypoF1b}) holds with
$\alpha_0=0$ and $n_1=n_2=2$. Therefore local existence  and
exponentiel decay for sufficiently small initial data hold for any $d\leq 2$.

\begin{Remark}{\rm
If $f$ is globally Lipschitz, i.e., $P_1=0$
and $P_2$ is a positive constant, we can alternatively use the theory from section \ref{pbform} and obtain exponential decay for small delays.
}\end{Remark}

\subsection{Predator-prey or two species competition systems with delays}\label{wavedamp}

Here we consider the semilinear diffusion system with time delays
 \begin{eqnarray}
& &u_{1,t}(x,t) =d_1\Delta u_1(x,t) +u_1(x,t)\Big(a_1+a_{11} u_1(x,t)\nonumber
\\
&&\hspace{3cm}+\sum_{j=1}^2a'_{1j} u_j(x,t-\tau_{1j})\Big)
\quad \mbox{\rm in}\ \Omega\times (0,+\infty),\label{C1.1pp}\\
& &u_{2,t} (x,t)=d_2\Delta u_2(x,t) +u_2(a_2+a_{22} u_2(x,t)
\nonumber
\\
&&\hspace{3cm}+\sum_{j=1}^2a'_{2j} u_j(x,t-\tau_{2j})\Big)
\quad \mbox{\rm in}\ \Omega\times (0,+\infty),\label{C1.1pp/2}\\
& & B_1u_1(x,t)=B_2u_2(x,t)=0\quad \mbox{\rm on}\ \Gamma\times (0,+\infty),\label{C1.2pp}\\
& &(u_1(x,t),u_2(x,t))=(u_{0,1}(x,t),u_{0,2}(x,t))\quad \mbox{\rm in}\ \Omega\times [-\tau ,0],\label{C1.3pp}
\end{eqnarray}
where the constants  $\tau_{ij} >0$ are the time delays, $\tau=\max \tau_{ij}$ and the initial datum $(u_{0,1},u_{0,2})^\top$ belongs to $C([0,\tau];L^2(\Omega )^2)$. The operator $B_i$ are in the form
\[
B_iu=\alpha_i \partial_n u_i+\beta_i u_i,
\]
with either $\alpha_i=0$ and $\beta_i=1$ corresponding to the case of Dirichlet boundary conditions
or $\alpha_i=1$ and $\beta_i\geq 0$ (with $\beta_i\in L^\infty(\partial\Omega)$)
 corresponding to the case of Neumann-Robin boundary conditions.
Here $d_i$ are positive constants, while $a_i, a_{ij}$ and $a'_{ij}$ are simply functions in $L^\infty(\Omega)$.

Systems for competitive species are studied in \cite{Pao:96} under some monotonicity properties on the nonlinear functions. Conditions for stability or instability of solutions are given independently of the time delays. Two species prey--predator models and competition system with delays are also analyzed in
 \cite{Ruan-Zhao:99}. By using the infinite--dimensional dissipative  system theory and comparison arguments, persistence criteria and also global extinction criteria are established.

System (\ref{C1.1pp})-(\ref{C1.3pp}) enters in the   framework of section \ref{secgeneralNL}. Indeed in such a situation, we take    $\caH =L^2(\Omega)^2$
and define $\caA$ as follows:
\[
D(\caA):=\{u=(u_1,u_2)^\top\in H^1(\Omega)^2: \Delta u_i\in L^2(\Omega) \hbox{ and  satisfying } B_iu_i= 0 \hbox{ on } \Gamma, i=1,2\},
\]
and
\[
\caA u=(d_1\Delta u_1-\varepsilon_1 u_1,d_2\Delta u_2-\varepsilon_2 u_2)^\top, \ \ \forall \
u\in D(\caA),
\]
where $\varepsilon_i>0$ if Neumann condition is imposed on $u_i$ (fixed later on), and $\varepsilon_i=0$ otherwise.
As before, $-\caA$ is a positive selfadjoint operator in $\caH$ with
\[
\caV=D((-\caA)^{\frac{1}{2}})= H^1_0(\Omega)^2,
\]
if Dirichlet boundary conditions are imposed on $u_1$ and $u_2$,
\[
\caV=D((-\caA)^{\frac{1}{2}})= H^1_0(\Omega)\times  H^1(\Omega)
\]
if Dirichlet boundary condition is imposed on $u_1$ and Neumann or Robin type on $u_2$,
\[
\caV=D((-\caA)^{\frac{1}{2}})= H^1(\Omega)\times  H^1_0(\Omega)
\]
if Dirichlet boundary condition is imposed on $u_2$ and Neumann or Robin type on $u_1$,
and finally
\[
\caV=D((-\caA)^{\frac{1}{2}})= H^1(\Omega)^2,
\]
otherwise.

We now distinguish four different cases:
\\
{\bf Case 1:} when Neumann boundary conditions are  imposed on both $u_1$ and $u_2$,
then  we take $\varepsilon_1=\varepsilon_2=\varepsilon$
and
the smallest eigenvalue $\lambda_1$ of $-\caA$ is $\varepsilon$.
\\
{\bf Case 2:} when Neumann boundary condition is imposed on $u_2$ and Dirichlet or Robin type on $u_1$,
then the smallest eigenvalue $\lambda_1$ of $-\caA$ is equal to $\min\{\mu_1,\varepsilon_2\}$,
where $\mu_1>0$ corresponds to the smallest eigenvalue of $-d_1\Delta$ with Dirichlet or Robin boundary conditions.
\\
{\bf Case 3:} when Neumann boundary condition is imposed $u_1$ and Dirichlet or Robin type on $u_2$,
then the smallest eigenvalue $\lambda_1$ of $-\caA$ is equal to $\min\{\mu_2,\varepsilon_1\}$,
where $\mu_2>0$ corresponds to the smallest eigenvalue of $-d_2\Delta$ with Dirichlet or Robin boundary conditions.
\\
{\bf Case 4:} when Dirichlet or Robin boundary conditions are imposed on both $u_1$ and  $u_2$,
then the smallest eigenvalue $\lambda_1$ of $-\caA$ is equal to $\min\{\mu_1,\mu_2\}$.

Without loss of generality we can suppose that
$\tau=\max\{\tau_{11}, \tau_{12}\}$
and therefore we set
\[
F_1(t, u,v)= ((a_1+\varepsilon_1) u_1+a_{11} u_1^2+a'_{11} u_1 v_1, 0)^\top,
\]
if $\tau_{11}>\tau_{12}$, and
\[
F_1(t, u,v)= ((a_1+\varepsilon_1) u_1+a_{11} u_1^2+a'_{12} u_1 v_2, 0)^\top,
\]
if $\tau_{11}<\tau_{12}$. Then  in the first case, we set $\tau_2=\tau_{21}$, $\tau_3=\tau_{12}$, $\tau_4=\tau_{22}$ and
for all $u=(u_1,u_2), v=(v_1,v_2)\in \mathbb{C}^2$:
\beqs
F_2(t,u,v)&=&(0,(a_2+\varepsilon_2)  u_2+a_{22} u_2^2+a'_{21} u_2v_1)^\top,\\
F_3(t,u,v)&=&(a'_{12} u_1  v_2, 0)^\top,\\
F_4(t,u,v)&=&(0,a'_{22} u_2v_2)^\top.
\eeqs
The second case is similar by simply changing $\tau_3$ and $F_3$ accordingly.
With these notations,  we see that problem (\ref{C1.1pp})-(\ref{C1.3pp}) can be written as (\ref{abstract}) with $I=2$, $U(t)=(u_1(\cdot, t), u_2(\cdot, t))^\top$
and $U_0(t)=(u_{0,1}(\cdot, t), u_{0,2}(\cdot, t))^\top$.
Consequently the next local existence result follows from Proposition \ref{WellP*}.

\begin{Proposition}\label{WellP*expp}
Assume that $d\leq 3$,
then there exists $\beta\in (0,\frac12)$
such that for any initial datum  $(u_{0,1},u_{0,2})^\top\in C ([0,\tau];\caV)\cap C^{0,\theta}([0, \tau], D((-\caA)^
\beta))$,  with $\theta=\min\{\beta, \frac{1}{2}-\beta\}$,  there exist a time $T_\infty \in (0, +\infty ]$ and a unique  solution
$(u_1,u_2)^\top\in C([0, T_\infty ), \caV)\cap C^1((0, T_\infty ), \caH)$  of problem $(\ref{C1.1pp})-(\ref{C1.3pp}).$
\end{Proposition}
\begin{proof}
It is a direct consequence of Proposition \ref{WellP*} since it is easy  to check that
$F_i$ satifies (\ref{hypoF1})
with  $n_1=n_2=1$
and
$C_0=\max_{i=1,2}\sup_\Omega |a_i+\varepsilon_i|$.
\end{proof}

As before global existence and exponential decay follow from Theorem \ref{Wellnonlineargeneral}.

\begin{Theorem}\label{Wellnonlineargeneral:expp}
Assume that $d\leq 3$ and that
\be\label{C1general:exppDirDir}
\sup_\Omega |a_i|<\lambda_1, i=1,2,
\ee
in case 1, that
\be\label{C1general:exppDirNeu}
\sup_\Omega |a_1|<\mu_1,  \ \sup_\Omega a_2<0,\  2\sup_\Omega a_2- \inf_\Omega a_2<2\mu_1,
\ee
in case 2, that
\be\label{C1general:exppNeuDir}
\sup_\Omega |a_2|<\mu_2,  \ \sup_\Omega a_1<0,\   2\sup_\Omega a_1- \inf_\Omega a_1<2\mu_2,
\ee
in case 3, and that
\be\label{C1general:exppNeuNeu}
\sup_\Omega  a_1 <0,  \ \sup_\Omega a_2<0,
\ee
in case 4.
Then there exist $K_0>0$ small enough and $\gamma_0<1$ (depending on $K_0$) such that for all
$K\in (0, K_0]$ and   $(u_{0,1},u_{0,2})^\top\in C([0,\tau];\caV )$ satisfying
\begin{equation}\label{assump:expp}
\|(u_{0,1},u_{0,2})^\top(t)\|_\caV<\gamma_0 K, \ \ \forall\  t\in [-\tau,0],
\end{equation}
problem $(\ref{C1.1})-(\ref{C1.3})$ has a global solution  $(u_1, u_2)^\top$ that satisfies the exponential decay estimate
 \begin{equation}\label{exponentialdbis:expp}
\|(u_1(\cdot, t), u_2(\cdot, t))^\top\|_\caH\le  M e^{-\tilde \omega t}
\quad \forall t\ge 0\,,
\end{equation}
for a   positive constant  $M$ depending on $u_0$ and a suitable  positive constant $\tilde\omega$.
 \end{Theorem}
\begin{proof}
It suffices to check that (\ref{H10bis}) and (\ref{C1general}) hold.
The first condition directly follows from Cauchy-Schwarz's inequality
where $C_1=C_0=\max_{i=1,2}\sup_\Omega |a_i+\varepsilon_i|$.
In the first and fourth cases, the second condition (\ref{C1general})
simply becomes either
(\ref{C1general:exppDirDir})  or (\ref{C1general:exppNeuNeu}).
Hence it remains to check that in cases 2 and 3, one can find $\varepsilon_1$ and $\varepsilon_2$
appropriately. By symmetry, it suffices to look at case 2. In such a situation,
(\ref{C1general}) becomes
\[
\max\{\sup_\Omega |a_1|, \sup_\Omega |a_2+\varepsilon_2|\}<\min \{\mu_1, \varepsilon_2\},
\]
that is clearly equivalent to
\beqs
\sup_\Omega |a_1|<\mu_1 \hbox{ and } \sup_\Omega |a_1|<\varepsilon_2,\\
\sup_\Omega |a_2+\varepsilon_2|<\mu_1 \hbox{ and } \sup_\Omega |a_2+\varepsilon_2|<\varepsilon_2.
\eeqs
These conditions hold if and only if there exists $\delta >0$ small enough such that
\beqs
 -\mu_1+\delta \leq a_1\leq \mu_1-\delta  \hbox{ and } -\varepsilon_2+\delta\leq a_1<\varepsilon_2-\delta  \hbox{ in } \Omega,\\
 -\mu_1+\delta\leq a_2+\varepsilon_2\leq \mu_1-\delta  \hbox{ and } -\varepsilon_2+\delta\leq  a_2+\varepsilon_2\leq \varepsilon_2-\delta
 \hbox{ in } \Omega.
\eeqs
By re-arranging those conditions, we find
\beq
\label{serge:13/03:2}
 -\mu_1+\delta \leq a_1\leq \mu_1-\delta  \hbox{ and } a_2\leq -\delta   \hbox{ in } \Omega,\\
\max\{a_1+\delta,-a_1+\delta, \frac{\delta-a_2}{2}, -a_2-\mu_1+\delta\}\leq
\varepsilon_2\leq \mu_1-\delta -a_2
 \hbox{ in } \Omega.
 \label{serge:13/03:3}
\eeq
A necessary condition to find such a $\varepsilon_2$ is that
\[
\max\{a_1+\delta,-a_1+\delta, \frac{\delta-a_2}{2}, -a_2-\mu_1+\delta\} \leq \mu_1-\delta -a_2
 \hbox{ in } \Omega,
 \]
which is indeed true if $\delta$ is small enough and if (\ref{serge:13/03:2}) holds. Therefore we fix
 \[
 \varepsilon_2=\inf_\Omega( \mu_1-\delta -a_2 )= \mu_1-\sup_\Omega a_2-\delta,
 \]
 with $\delta>0$ small enough so that this quantity is positive.
 We then easily check that (\ref{serge:13/03:3}) holds
 under the assumption (\ref{serge:13/03:2}) and the condition
 \[
 3\delta+2\sup_\Omega a_2- \inf_\Omega a_2\leq 2\mu_1.
 \]
Hence this  condition and  (\ref{serge:13/03:2}) hold for $\delta$ small enough if and only if
 (\ref{C1general:exppDirNeu}) is valid.
 \end{proof}

\begin{Remark}{\rm
Clearly, our approach can be used to study a general system of $N$-competing species with delays
like the system (1.3) of \cite{Pao:04} since the nonlinear terms appearing in this system
are similar to the ones from system (\ref{C1.1pp})-(\ref{C1.3pp}). We then let the details to the reader.
}\end{Remark}


\begin{thebibliography}{10}


\bibitem{AlNP}
{\sc F.~Alabau-Boussouira, S.~Nicaise and C.~Pignotti.}
\newblock Exponential stability of the wave equation with memory
 and time delay.
\newblock{\em New Prospects in Direct, Inverse and Control Problems for Evolution Equations}, Springer Indam Ser., 10:1--22, 2014.




\bibitem{ANP10}
{\sc K.~Ammari, S. Nicaise and C.~Pignotti.}
\newblock Feedback boundary stabilization of wave equations with interior delay.
\newblock {\em Systems and Control Lett.},  59:623--628, 2010.


\bibitem{Batkai}
{\sc A.~B\'{a}tkai and S.~Piazzera.}
\newblock {\em Semigroups for delay equations,}
\newblock Research Notes in Mathematics, 10. AK Peters, Ltd., Wellesley, MA, 2005.

\bibitem{BLR}
{\sc C.~Bardos, G.~Lebeau and J.~Rauch.}
\newblock {Sharp sufficient conditions for the observation, control and
stabilization of waves from the boundary.}
\newblock {\em SIAM J. Control Optim.}, 30:1024--1065, 1992.

\bibitem{Dafermos}
{\sc C.~M.~Dafermos.}
\newblock Asymptotic stability in viscoelasticity.
\newblock {\em Arch. Rational Mech. Anal.}, 37:297--308, 1970.


\bibitem{Datko}
{\sc R.~Datko.}
\newblock Not all feedback stabilized hyperbolic systems are robust
with respect to small time delays in their feedbacks.
\newblock {\em SIAM J. Control Optim.}, 26:697--713, 1988.


\bibitem{DLP}
{\sc R.~Datko, J.~Lagnese and M.~P.~Polis.}
\newblock An example on the effect of time delays in boundary feedback stabilization of
wave equations.
\newblock  {\em SIAM J. Control Optim.}, 24:152--156, 1986.


\bibitem{Freedman-Zhao:97}
{\sc H.~I.~ Freedman and X.~Q.~Zhao.}
\newblock Global asymptotics in some quasimonotone reaction-diffusion systems
  with delays.
\newblock {\em J. Differential Equations}, 137(2):340--362, 1997.

\bibitem{Friesecke}
{\sc G.~Friesecke.}
\newblock  Convergence to equilibrium for delay-diffusion equations with small delay.
\newblock  {\em J. Dynam. Differential Equations},  5:89-–103, 1993.

\bibitem{Pata}
{\sc C.~Giorgi, J.~E.
~Mu\~{n}oz Rivera and V.~Pata.}
\newblock Global attractors for a semilinear hyperbolic equation in viscoelasticity.
\newblock {\em J. Math. Anal. Appl.}, 260:83--99, 2001.




\bibitem{grisvard:85a}
{\sc P.~Grisvard.}
\newblock {\em Elliptic problems in nonsmooth domains}, volume~24 of {\em
  Monographs and Studies in Mathematics}.
\newblock Pitman, Boston--London--Melbourne, 1985.

\bibitem{guesmia}
{\sc A.~Guesmia.}
\newblock Well-posedness and exponential stability of an abstract evolution equation with infinite memory and time delay.
\newblock {\em IMA J. Math. Control Inform.,} 30:507--526, 2013.

\bibitem{Inoueetall:77}
{\sc A.~Inoue, T.~Miyakawa and K.~Yoshida.}
\newblock Some properties of solutions for semilinear heat equations with time
  lag.
\newblock {\em J. Differential Equations}, 24(3):383--396, 1977.

\bibitem{kato67}
{\sc T.~Kato.}
\newblock {\em
Perturbation theory for linear operators}.
{Springer-Verlag},
 {Berlin},  {1976}.
 
 \bibitem{Komornikbook}
{\sc V.~Komornik.}
\newblock {\em Exact controllability and stabilization, the multiplier method},
  volume~36 of {\em RMA}.
\newblock Masson, Paris, 1994.



\bibitem{Lighbourne:81}
{\sc J.~H.~Lightbourne and S.~M.~Rankin.}
\newblock Global existence for a delay differential equation.
\newblock {\em J. Differential Equations}, 40(2):186--192, 1981.

\bibitem{Liu:02}
{\sc W.~Liu.}
\newblock Asymptotic behavior of solutions of time-delayed Burgers' equation.
\newblock {\em Discrete Contin. Dyn. Syst. Ser. B}, 2(1):47--56, 2002.




\bibitem{Memory:91}
{\sc M.~C.~Memory.}
\newblock Stable and unstable manifolds for partial functional-differential
  equations.
\newblock {\em Nonlinear Anal.}, 16(2):131--142, 1991.


\bibitem{NPSicon06}
{\sc S.~Nicaise and C.~Pignotti.}
\newblock Stability and instability results of the wave equation with a delay term in the boundary
or internal feedbacks.
\newblock {\em SIAM J. Control Optim.}, 45:1561--1585, 2006.

\bibitem{MCSS}
{\sc S.~Nicaise and C.~Pignotti.}
\newblock Stabilization of second-order evolution equations with time delay.
\newblock {\em Math. Control Signals Systems}, 26:563--588, 2014.

\bibitem{JEE15}
{\sc S.~Nicaise and C.~Pignotti.}
\newblock Exponential stability of abstract evolution equations with time delay.
\newblock {\em J. Evol. Equ.},  15:107--129, 2015.

\bibitem{Oliva:99}
{\sc S.~M.~Oliva.}
\newblock
Reaction-diffusion equations with nonlinear boundary delay.
\newblock {\em J. Dynam. Differential Equations}, 11:279--296, 1999.


\bibitem{Pao:96}
{\sc C.~V.~Pao.}
\newblock Dynamics of nonlinear parabolic systems with time delays.
\newblock {\em J. Math. Anal. Appl.}, 198(3):751--779, 1996.

\bibitem{Pao:04}
{\sc C.~V.~Pao.}
\newblock Global asymptotic stability of {L}otka-{V}olterra competition systems
  with diffusion and time delays.
\newblock {\em Nonlinear Anal. Real World Appl.}, 5(1):91--104, 2004.



\bibitem{pazy}
{\sc A.~Pazy.}
\newblock {\em Semigroups of linear operators and applications to partial differential equations}, Vol. 44 of {\em Applied Math. Sciences.} Springer-Verlag, New York, 1983.

\bibitem{SCL12}
{\sc C.~Pignotti.}
\newblock A note on stabilization of locally damped wave equations with time delay.
\newblock {\em Systems and Control Lett.}, 61:92--97,  2012.



\bibitem{Ruan-Zhao:99}
{\sc S.~Ruan and X.~Q.~Zhao.}
\newblock Persistence and extinction in two species reaction-diffusion systems
  with delays.
\newblock {\em J. Differential Equations}, 156(1):71--92, 1999.


\bibitem{Said}
{\sc B.~Said-Houari and A.~Soufyane.}
\newblock Stability result of the Timoshenko system with delay and boundary feedback.
\newblock {\em IMA J. Math. Control Inform.,} 29:383--398, 2012.






\bibitem{XYL}
{\sc G.~Q.~Xu, S.~P.~Yung and L.~K.~Li.}
\newblock Stabilization of wave systems with input delay in the boundary control.
\newblock {\em ESAIM Control Optim. Calc. Var.}, 12(4):770--785,  2006.






\bibitem{zuazua}
{\sc E.~Zuazua.}
\newblock Exponential decay for the semi-linear wave equation with locally distributed damping.
\newblock {\em Comm. Partial Differential Equations}, 15:205--235, 1990.

\end{thebibliography}

 \end{document}